\documentclass[11pt]{article}
\usepackage{amsthm,geometry,amssymb,amsmath,enumerate,float,tikz,cite,algorithm2e,setspace}
\newtheorem{theorem}{Theorem}[section]
\newtheorem{lemma}{Lemma}[section]
\newtheorem{claim}{Claim}

\usepackage{lineno}

\title{The matching number of tree and bipartite degree sequences}
\author{F. Bock\and D. Rautenbach}
\date{}

\begin{document}
\onehalfspace

\maketitle

\begin{center}
{\small 
Institute of Optimization and Operations Research, Ulm University, Ulm, Germany\\
\texttt{$\{$felix.bock,dieter.rautenbach$\}$@uni-ulm.de}\\[3mm]
}
\end{center}

\begin{abstract}
We study the possible values of the matching number 
among all trees with a given degree sequence
as well as 
all bipartite graphs with a given bipartite degree sequence.
For tree degree sequences, we obtain closed formulas for the possible values.
For bipartite degree sequences,
we show the existence of realizations with a restricted structure,
which allows to derive an analogue of the Gale-Ryser Theorem
characterizing bipartite degree sequences.
More precisely, we show that a bipartite degree sequence
has a realization with a certain matching number 
if and only if 
a cubic number of inequalities 
similar to those in the Gale-Ryser Theorem
are satisfied.
For tree degree sequences as well as for bipartite degree sequences,
the possible values of the matching number form intervals.
\end{abstract}
{\small 
\begin{tabular}{lp{13cm}}
{\bf Keywords}: Matching; matching number; degree sequence; tree; bipartite graph
\end{tabular}
}

\pagebreak

\section{Introduction}

In the present paper we study the possible values of the matching number among all trees with a given degree sequence
as well as 
all bipartite graphs with a given bipartite degree sequence.

We consider finite, simple, and undirected graphs, and use standard terminology.
For a graph $G$, let $\nu(G)$, $\tau(G)$, and $\alpha(G)$ be the matching number, the vertex cover number, and the independence number of $G$, respectively. By the classical results of Gallai \cite{ga2}, K\H{o}nig \cite{ko}, and Egerv\'ary \cite{eg}, 
the sum $\tau(G)+\alpha(G)$ equals the order $n(G)$ of $G$ for every graph $G$,
and $\nu(G)$ equals $\tau(G)$ for every bipartite graph $G$,
in particular, for every tree.
The {\it degree sequence} of a graph $G$ is the nonincreasing sequence $d(G)$ 
of the degrees of the vertices of $G$,
that is, if $G$ has vertices $v_1,\ldots,v_n$ and $d_G(v_1)\geq\ldots\geq d_G(v_n)$, 
then $d(G)=(d_G(v_1),\ldots,d_G(v_n))$.
Every graph is a {\it realization} of its degree sequence.
For a finite sequence $d$ of integers and an integer $i$, 
let $n(d)$ be the number of elements of $d$, and
let $n_i(d)$ be the number of elements of $d$ that equal $i$.
A {\it tree degree sequence} is a degree sequence of some tree.
For an integer $k$, let $[k]$ be the set of positive integers at most $k$.

For a degree sequence $d$, 
a class ${\cal C}$ of graphs, 
a graph invariant $\eta$, and an optimization goal
${\rm opt}\in \{ \min, \max\}$, let 
$$\eta^{\cal C}_{\rm opt}(d)={\rm opt}\big\{ \eta(G):G\in {\cal C}\mbox{ and }d(G)=d\big\},$$
that is, 
$\eta^{\cal C}_{\rm opt}(d)$ captures an extremal value of the invariant $\eta$ among all realizations $G$ of $d$ that belong to ${\cal C}$.
Let ${\cal G}$ and ${\cal T}$ be the classes of all graphs and trees, respectively.

Rao \cite{ra} showed that $\alpha^{\cal G}_{\max}(d)$ 
can be determined efficiently for every given $d$ (cf. also \cite{ra2,kele,yi}).
Similarly, Gentner et al. determined $\alpha^{\cal T}_{\min}(d)$ \cite{gehera1} and $\alpha^{\cal T}_{\max}(d)$ \cite{gehera2} 
for every given tree degree sequence $d$.
By the results of Gallai, K\H{o}nig, and Egerv\'ary,
$$\nu^{\cal T}_{\max}(d)=\tau^{\cal T}_{\max}(d)=n(d)-\alpha^{\cal T}_{\min}(d)
\,\,\,\,\,\,\,\,\mbox{ and }\,\,\,\,\,\,\,\,
\nu^{\cal T}_{\min}(d)=\tau^{\cal T}_{\min}(d)=n(d)-\alpha^{\cal T}_{\max}(d).$$
As our first contribution in Section \ref{sectrees},
we strengthen these results from \cite{gehera1,gehera2} giving simpler proofs.
We not only recover the extremal values 
but also show that all intermediate integers are independence numbers/matching numbers of tree realizations.

In Section \ref{secbip}, we consider bipartite graphs.
We prove a version of the well known Gale-Ryser theorem \cite{ga,ry}
that allows to determine the possible values of the matching number efficiently.
For a bipartite graph $G$ with fixed partite sets $A$ and $B$,
the {\it bipartite degree sequence} $d_{\rm bip}(G)$
is the pair $(d_A,d_B)$,
where 
$d_A$ is the nonincreasing sequence of the degrees of the vertices in $A$, and
$d_B$ is the nonincreasing sequence of the degrees of the vertices in $B$,
that is,
if $A=\{ v_1,\ldots,v_n\}$,
$d_G(v_1)\geq \ldots \geq d_G(v_n)$,
$B=\{ w_1,\ldots,w_m\}$, and 
$d_G(w_1)\geq \ldots \geq d_G(w_m)$,
then 
$$d_{\rm bip}(G)=((d_G(v_1),\ldots,d_G(v_n)),(d_G(w_1),\ldots,d_G(w_m))).$$
Every bipartite graph is a {\it realization} of its bipartite degree sequence.
Note that fixing the degrees in the two partite sets
easily allows to construct pairs of sequences 
$((a_1,\ldots,a_n),(b_1,\ldots,b_m))$
that are no bipartite degree sequences
even though the nonincreasing reorderings of $(a_1,\ldots,a_n,b_1,\ldots,b_m)$
have realizations that are bipartite.

The Gale-Ryser Theorem \cite{ga,ry} states that 
$((a_1,\ldots,a_n),(b_1,\ldots,b_m))$,
where all $a_i$ and $b_j$ are nonnegative integers,
is a bipartite degree sequence if and only if 
\begin{eqnarray}\label{e0}
\sum\limits_{i=1}^na_i=\sum\limits_{j=1}^mb_i
\,\,\,\,\,\,\,\,\,\mbox{ and }\,\,\,\,\,\,\,\,\,
\sum\limits_{i=1}^ka_i\leq \sum\limits_{j=1}^m\min\{ b_j,k\}
\mbox{ for every $k$ in $[n]$},
\end{eqnarray}
that is, the Gale-Ryser Theorem reduces the existence 
of a realization of a bipartite degree sequence 
to a linearly many inequalities.

For a bipartite degree sequence $(d_A,d_B)$
and an integer $\nu$ 
that is the matching number of some realization of $(d_A,d_B)$,
we establish the existence of a realization 
with a well specified maximum matching
and minimum vertex cover.
This allows to apply network flows to the considered problems,
and to reduce the existence of realizations of $(d_A,d_B)$ 
with a given matching number
to a cubic number of inequalities similar to (\ref{e0}).

\section{Trees}\label{sectrees}

It is well known that a sequence $(d_1,\ldots,d_n)$ of at least $2$ nonnegative integers is a tree degree sequence if and only if 
\begin{eqnarray}\label{e1}
n_1(d)&=&2+\sum\limits_{i:d_i\geq 2}(d_i-2).
\end{eqnarray}

\begin{lemma}\label{lemma1}
If $d$ is a tree degree sequence with $n(d)\geq 3$, then 
$$\nu^{\cal T}_{\max}(d)=\min\left\{ \left\lfloor \frac{n(d)}{2}\right\rfloor,n(d)-n_1(d)\right\}.$$
\end{lemma}
\begin{proof}
Let $n=n(d)$ and $n_1=n_1(d)$.

Clearly, every matching in a graph of order $n$ contains at most $n/2$ edges,
which implies $\nu^{\cal T}_{\max}(d)\leq \lfloor n/2\rfloor$.
Furthermore, every edge of a tree of order $n$ at least $3$ is incident with a vertex of degree at least $2$,
which implies $\nu^{\cal T}_{\max}(d)\leq n-n_1$.
Altogether, we obtain $\nu^{\cal T}_{\max}(d)\leq \nu_{\max}:=\min\{ \lfloor n/2\rfloor,n-n_1\}$.
In order to complete the proof, 
we prove, by induction on $n_2(d)$,
that there is a tree $T$ with $d(T)=d$ and $\nu(T)\geq \nu_{\max}$.

First, let $n_2(d)=0$.
Let the tree $T$ of order $n$ arise by attaching leaves to a path $P:v_1\ldots v_{n-n_1}$ in such a way that $d_T(v_i)=d_i$ for every $i$ in $[n-n_1]$.
Clearly, $d(T)=d$. 
Furthermore, since every vertex $v_i$ of $P$ is adjacent to a leaf $w_i$ outside of $P$, 
the set $M=\{ v_iw_i:i\in [n-n_1]\}$ is matching in $T$.
Since $|M|=\nu_{\max}\geq \nu(T)$, 
it follows that $\nu(T)=|M|=\nu_{\max}$.

Now, let $n_2(d)>0$.
By (\ref{e1}), the sequence $d'$ that arises from $d$ by removing one element of $d$ that equals $2$, 
is a tree degree sequence.
By induction, there is a tree $T'$ with $d(T')=d'$ and 
$$\nu(T')\geq
\min\left\{ \left\lfloor \frac{n(d')}{2}\right\rfloor,n(d')-n_1(d')\right\}=
\min\left\{ \left\lfloor \frac{n-1}{2}\right\rfloor,(n-1)-n_1\right\}.$$
If $\nu(T')=(n-1)/2$, 
then $n$ is odd, and subdividing some edge of $T'$ yields a tree $T$ with $d(T)=d$, 
and $\nu(T)\geq \nu(T')=(n-1)/2=\lfloor n/2\rfloor\geq \nu_{\max}$.
Hence, we may assume that $\nu(T')<(n-1)/2$, 
which implies the existence of a maximum matching $M'$ in $T'$,
and a vertex $u$ of $T'$ that is not incident with an edge in $M'$.
Let $w$ be a neighbor of $u$ in $T'$, and 
let $T$ arise from $T'$ by subdividing the edge $uw$ with a new vertex $v$.
Clearly, $d(T)=d$, and $M=M'\cup \{ uv\}$ is a matching in $T$.
Now, $\nu(T')\geq |M|\geq\min\{ \lfloor (n-1)/2\rfloor,(n-1)-n_1\}+1\geq \nu_{\max}$,
which completes the proof.
\end{proof}

\begin{lemma}\label{lemma2}
If $(d_1,\ldots,d_n)$ is a tree degree sequence with $n\geq 3$, 
$V=\{ v_1,\ldots,v_n\}$, and $X\subseteq V$ are such that 
\begin{enumerate}[(i)]
\item $d_i>1$ for every $v_i$ in $X$,
\item $|X|\leq n/2$, and 
\item $\sum\limits_{i:v_i\in X}d_i\geq \sum\limits_{i:v_i\in V\setminus X}d_i$,
\end{enumerate}
then there is a tree $T$ with $V(T)=V$ such that 
$d_T(v_i)=d_i$ for every $v_i$ in $V$, and 
$X$ is a minimum vertex cover in $T$.
\end{lemma}
\begin{proof}
Let $d=(d_1,\ldots,d_n)$ and $s=\sum\limits_{i:v_i\in V\setminus X}(d_i-1)$.
We prove the statement by induction on $s$.

First, let $s=0$.
In this case, $d_i=1$ for every $v_i$ in $V\setminus X$.
By (i), this implies $|X|=n-n_1(d)$.
By (ii), $|X|=\min\{ \lfloor n/2\rfloor,n-n_1(d)\}$, and Lemma \ref{lemma1} implies that
there is a tree $T$ with $V(T)=V$, $d_T(v_i)=d_i$ for every $i$ in $[n]$, and $\tau(T)=\nu(T)=|X|$.
Since $n\geq 3$, every edge in $T$ is incident with a vertex in $X$,
which implies that $X$ is a minimum vertex cover in $T$.

Now, let $s>0$.
This implies that $d_p>1$ for some $v_p$ in $V\setminus X$.
By symmetry, we may assume that $d_{p+1}<d_p$.
Let $v_q$ in $X$ be such that $d_q=\min\{ d_i:v_i\in X\}$.
Since $d$ is a tree degree sequence, $n_1(d)\geq d_q$.
Let $d'$ arise from $d$
\begin{itemize}
\item by decreasing $d_p$ by $1$, and
\item by removing $d_q$ as well as $d_q-1$ elements that equal $1$.
\end{itemize}
Since $d_{p+1}<d_p$, the sequence $d'$ is nonincreasing.

For convenience, let $d'=(d'_1,\ldots,d'_{q-1},d'_{q+1},\ldots,d_{n'+1}')$.

If $d_p\geq 3$, then, by (\ref{e1}), 
$$n_1(d')=n_1(d)-(d_q-1)
=2+\sum\limits_{i:d_i\geq 2}(d_i-2)-(d_q-2)-1
=2+\sum\limits_{i:d'_i\geq 2}(d'_i-2),$$ and, if $d_p=2$,
then 
$$n_1(d')=n_1(d)-(d_q-1)+1
=2+\sum\limits_{i:d_i\geq 2}(d_i-2)-(d_q-2)
=2+\sum\limits_{i:d'_i\geq 2}(d'_i-2).$$
By (\ref{e1}), it follows that $d'$ is the degree sequence of a tree of order $n'$ at least $2$.
If $n'=2$, then $d_p=2$, $X=\{ v_q\}$, and 
$d_q=\sum\limits_{i:v_i\in X}d_i<\sum\limits_{i:v_i\in V\setminus X}d_i$,
which is a contradiction. Hence, $n'\geq 3$.

Let $V'=V\setminus (\{ v_q\} \cup \{ v_{n-d_q+2},\ldots,v_n\})$ and $X'=X\setminus \{ v_q\}$.

In order to apply induction, we verify the properties (i), (ii), and (iii) for $d'$, $V'$, and $X'$.

Trivially, $d'_i>1$ for every $v_i$ in $X'$.

If $d_q=2$, then, by (i), 
$|X'|=|X|-1\leq |V\setminus X|-1=|V'\setminus X'|,$
which implies $|X'|\leq n'/2$.
If $d_q\geq 3$, then, by (\ref{e1}) and the choice of $q$, 
we obtain $n_1(d)\geq 2+|X|(d_q-2)$,
which implies 
$$|V\setminus X|\geq 1+n_1(d)\geq 3+|X|(d_q-2)\geq |X|+d_q.$$
We obtain 
$|X'|=|X|-1\leq |V\setminus X|-d_q-1=|V'\setminus X'|$,
which implies $|X'|\leq n'/2$.
Altogether, $|X'|\leq n'/2$ follows in both cases.

Finally, 
$$\sum\limits_{i:v_i\in X'}d'_i
=\sum\limits_{i:v_i\in X}d_i-d_q
\geq \sum\limits_{i:v_i\in V\setminus X}d_i-d_q
=\sum\limits_{i:v_i\in V'\setminus X'}d_i'.$$
Since $\sum\limits_{i:v_i\in V'\setminus X'}(d_i'-1)<s$, we obtain, by induction, 
that there is a tree $T'$ with $V(T')=V'$ such that 
$d_{T'}(v_i)=d_i'$ for every $v_i$ in $V'$, and 
$X'$ is a minimum vertex cover in $T'$.
Let $M'$ be a maximum matching in $T'$.
Let $T$ arise by adding the edge $v_pv_q$ to the union of $T'$ and a star with center vertex $v_q$
and leaves $v_{n-d_q+2},\ldots,v_n$.
Clearly, $T$ is a tree with $V(T)=V$ such that 
$d_T(v_i)=d_i$ for every $v_i$ in $V$, and 
$X=X'\cup \{ v_q\}$ is a vertex cover in $T$.
Since $M=M'\cup \{ v_qv_n\}$ is a matching in $T$ with $|X|=|M|$,
it follows that $X$ is a minimum vertex cover in $T$,
which completes the proof.
\end{proof}

\begin{theorem}\label{theorem1}
Let $(d_1,\ldots,d_n)$ be a a tree degree sequence with $n\geq 3$, and let $\nu$ be an integer.

There is a tree $T$ with $d(T)=(d_1,\ldots,d_n)$ and $\nu(T)=\nu$
if and only if 
\begin{eqnarray}\label{e2}
\min\left\{ k\in \mathbb{N}:\sum\limits_{i=1}^k d_i\geq n-1\right\}
\leq \nu
\leq 
\min\left\{ \left\lfloor \frac{n}{2}\right\rfloor,n-n_1(d)\right\}.
\end{eqnarray}
\end{theorem}
\begin{proof}
Let $T$ be a tree with $d(T)=(d_1,\ldots,d_n)$ and $\nu=\nu(T)$.
By Lemma \ref{lemma1}, $\nu\leq \min\{ \lfloor n/2\rfloor,n-n_1(d)\}$.
If $X$ is a minimum vertex cover in $T$, then $|X|=\nu$.
Since $V(T)\setminus X$ is independent, 
we obtain
$$\sum\limits_{i=1}^{\nu}d_i
\geq \sum\limits_{u\in X}d_T(u)
\geq \sum\limits_{u\in V(T)\setminus X}d_T(u)
\geq \sum\limits_{i=\nu+1}^{n}d_i,$$
which, using $\sum\limits_{i=1}^nd_i=2(n-1)$, implies $\sum\limits_{i=1}^{\nu}d_i\geq n-1$, 
and thus $\nu\geq \min\left\{ k\in \mathbb{N}:\sum\limits_{i=1}^kd_i\geq n-1\right\}$.

Now, let the integer $\nu$ satisfy (\ref{e2}).
Let $V=\{ v_1,\ldots,v_n\}$ and $X=\{ v_1,\ldots,v_{\nu}\}$.
Since $\nu\leq n-n_1(d)$, we have $d_i>1$ for every $v_i\in X$.
Since $\nu\leq \lfloor n/2\rfloor$, we have $|X|\leq n/2$.
Since $\nu\geq \min\left\{ k\in \mathbb{N}:\sum\limits_{i=1}^k d_i\geq n-1\right\}$ and $\sum\limits_{i=1}^nd_i=2(n-1)$,
we have $\sum\limits_{i:v_i\in X}d_i\geq \sum\limits_{i:v_i\in V\setminus X}d_i$.
By Lemma \ref{lemma2}, there is a tree $T$ with $d(T)=(d_1,\ldots,d_n)$ and $\nu(T)=\nu$,
which completes the proof.
\end{proof} 

\section{Bipartite graphs}\label{secbip}

Our first result in this section establishes the existence of a realization 
of a bipartite degree sequence
with a well specified maximum matching
and minimum vertex cover.

\begin{theorem}\label{theorem2}
If $((a_1,\ldots,a_n),(b_1,\ldots,b_m))$ is a bipartite degree sequence 
that has a realization with matching number $\nu$, 
then there is a realization $G$ 
with partite sets $A=\{ v_1,\ldots,v_n\}$ and $B=\{ w_1,\ldots,w_m\}$,
and an integer $k$ with $0\leq k\leq \nu$ such that 
\begin{enumerate}[(i)]
\item $\{ v_iw_{\nu-i+1}:i\in [\nu]\}$ is a maximum matching in $G$, and
\item $\{ v_i:i\in [k]\}\cup \{ w_j:j\in [\nu-k]\}$ is a minimum vertex cover in $G$.
\end{enumerate}
\end{theorem}
\begin{proof} 
Let $d=(d_A,d_B)=((a_1,\ldots,a_n),(b_1,\ldots,b_m))$.
Let $G$ be a realization of $d$ with matching number $\nu$ and 
partite sets $A=\{ v_1,\ldots,v_n\}$ and $B=\{ w_1,\ldots,w_m\}$,
that is, $d_G(v_i)=a_i$ for every $i$ in $[n]$, and 
$d_G(w_j)=b_j$ for every $i$ in $[m]$.
Let $M$ be a maximum matching in $G$,
and let $X$ be a minimum vertex cover in $G$.
Since $|M|=|X|$, for every edge $e$ in $M$, the set $X$ contains exactly one vertex incident with $e$,
and every vertex in $X$ is incident with an edge in $M$.

Let $X_A=X\cap A$, 
$\overline{X}_A=A\setminus X$,
$X_B=X\cap B$, 
$\overline{X}_B=B\setminus X$,
and $k=|X_A|$.
Let $V(M)$ be the set of vertices of $G$ that are incident with an edge in $M$.
For the tuple $g=(G,M,X)$, let
\begin{itemize}
\item $s_1(g)$ be the number of pairs $(v_r,v_s)$ 
with $v_r\in X_A$, $v_s\in \overline{X}_A$, and $a_r<a_s$,
\item $s_2(g)$ be the number of pairs $(w_r,w_s)$ 
with $w_r\in X_B$, $w_s\in \overline{X}_B$, and $b_r<b_s$,
\item $s_3(g)$ be the number of pairs $(v_r,v_s)$ 
with $v_r\in V(M)\cap \overline{X}_A$, $v_s\in \overline{X}_A\setminus V(M)$, and $a_r<a_s$,
\item $s_4(g)$ be the number of pairs $(w_r,w_s)$ 
with $w_r\in V(M)\cap \overline{X}_B$, $w_s\in \overline{X}_B\setminus V(M)$, and $b_r<b_s$,
\item $s_5(g)$ be the number of pairs $(v_r,v_s)$ with $v_r,v_s\in X_A$ 
such that $a_r<a_s$ and $b_x<b_y$, where $w_x$ and $w_y$ are such that $v_rw_x,v_sw_y\in M$, 
\item $s_6(g)$ be the number of pairs $(w_r,w_s)$ with $w_r,w_s\in X_B$ 
such that $a_x<a_y$ and $b_r<b_s$, where $v_x$ and $v_y$ are such that $v_xw_r,v_yw_s\in M$, and
\item $s(g)=(s_1(g),s_2(g),s_3(g),s_4(g),s_5(g),s_6(g))$.
\end{itemize}
We assume that $g=(G,M,X)$ is chosen in such a way that $s(g)$ is lexicographically minimal.
In order to complete the proof, it suffices to show that $s(g)=(0,0,0,0,0,0)$.

\begin{claim}\label{claim1}
$s_1(g)=s_2(g)=0$.
\end{claim}
\begin{proof}[Proof of Claim \ref{claim1}]
Suppose, for a contradiction, that $s_1(g)>0$.
Let $(v_r,v_s)$ is as in the definition of $s_1(g)$.
Let $N_r$ be the set of neighbors of $v_r$ in $\overline{X}_B$.
Since $a_r<a_s$, and $v_s$ has no neighbor in $\overline{X}_B$,
there is a set $N_s$ of $|N_r|$ neighbors of $v_s$ in $X_B$ that are not adjacent to $v_r$.
If $v_s\in V(M)$, $w_y$ is such that $v_sw_y\in M$, 
and $v_r$ is not adjacent to $w_y$,
then let $w_y$ belong to $N_s$.
Let $G'$ arise from $G$ by
\begin{itemize}
\item removing all edges between $v_s$ and $N_s$, and between $v_r$ and $N_r$, and
\item adding all edges between $v_s$ and $N_r$, and between $v_r$ and $N_s$.
\end{itemize}
Clearly, $G'$ is a realization of $d$, and all vertices have the same degrees in $G'$ as in $G$. 
By construction, $X'=(X\setminus \{ v_r\})\cup \{ v_s\}$ is a vertex cover in $G'$.
Let $w_x$ be such that $v_rw_x\in M$.
By construction,
$$
M'=
\begin{cases}
(M\setminus \{ v_rw_x\})\cup \{ v_sw_x\} & \mbox{, if $v_s\not\in V(M)$, and}\\
(M\setminus \{ v_rw_x,v_sw_y\})\cup \{ v_sw_x,v_rw_y\} & \mbox{, if $v_s\in V(M)$}
\end{cases}$$
is a matching in $G'$.
Since $|X'|=|X|=|M|=|M'|$,
the set $X'$ is a minimum vertex cover in $G'$, and 
$M'$ is a maximum matching in $G'$.
Since $s_1((G',M',X'))<s_1(g)$, we obtain a contradiction to the choice of $g$,
which implies $s_1(g)=0$.
Since $s_2((G',M',X'))=s_2(g)$, 
that is, $s_2$ is not affected by the modifications,
we obtain, by symmetry, $s_2(g)=0$,
which completes the proof of the claim. 
\end{proof} 

\begin{claim}\label{claim2}
$s_3(g)=s_4(g)=0$.
\end{claim}
\begin{proof}[Proof of Claim \ref{claim2}]
Suppose, for a contradiction, that $s_3(g)>0$.
Let $(v_r,v_s)$ is as in the definition of $s_3(g)$.
Let $w_x$ be such that $v_rw_x\in M$.

If $v_s$ is adjacent to $w_x$, 
then $M'=(M\setminus \{ v_rw_x\})\cup \{ v_sw_x\}$ is a maximum matching in $G$,
$s_1((G,M',X))=s_2((G,M',X))=0$, 
$s_3((G,M',X))<s_3(g)$, and
$s_4((G,M',X))=s_4(g)$,
contradicting the choice of $g$.
Hence, $v_s$ is not adjacent to $w_x$.
Since $a_r<a_s$, and all neighbors of $v_s$ are in $X_B$,
there is a neighbor $w_y$ of $v_s$ in $X_B$ that is not adjacent to $v_r$.
Let $G'$ arise from $G$ by removing the edges $v_rw_x$ and $v_sw_y$,
and adding the edges $v_rw_y$ and $v_sw_x$,
and let $M'=(M\setminus \{ v_rw_x\})\cup \{ v_sw_x\}$.
By construction, $X$ is a minimum vertex cover of $G'$,
and $M'$ is a maximum matching in $G'$.
Since 
$s_1((G',M',X))=s_2((G',M',X))=0$, 
$s_3((G',M',X))<s_3(g)$, and
$s_4((G',M',X))=s_4(g)$,
we obtain a contradiction to the choice of $g$.

Altogether, we obtain $s_3(g)=0$.

In both cases, $s_4$ is not affected by the modifications.
By symmetry, this implies $s_4(g)=0$,
which completes the proof of the claim. 
\end{proof} 

\begin{claim}\label{claim3}
$s_5(g)=s_6(g)=0$.
\end{claim}
\begin{proof}[Proof of Claim \ref{claim2}]
Suppose, for a contradiction, that $s_5(g)>0$.
Let $(v_r,v_s)$, $w_x$, and $w_y$ be as in the definition of $s_5(g)$.

First, we assume that the edges $v_rw_y$ and $v_sw_y$ belong to $G$.
Now, $M'=(M\setminus \{ v_rw_x,v_sw_y\})\cup \{ v_sw_x,v_rw_y\}$ is a maximum matching in $G$,
$s_1((G,M',X))=s_2((G,M',X))=s_3((G,M',X))=s_4((G,M',X))=0$, 
$s_5((G,M',X))<s_5(g)$, and
$s_6((G,M',X))=s_6(g)$,
contradicting the choice of $g$.

Next, we assume that neither of the edges $v_rw_y$ and $v_sw_y$ belongs to $G$.
Let $G'$ arise from $G$ by removing the edges $v_rw_x$ and $v_sw_y$,
and adding the edges $v_rw_y$ and $v_sw_x$,
and let $M'=(M\setminus \{ v_rw_x,v_sw_y\})\cup \{ v_sw_x,v_rw_y\}$.
By construction, $X$ is a minimum vertex cover of $G'$,
and $M'$ is a maximum matching in $G'$.
Since 
$s_1((G,M',X))=s_2((G,M',X))=s_3((G,M',X))=s_4((G,M',X))=0$, 
$s_5((G',M',X))<s_5(g)$, and
$s_6((G',M',X))=s_6(g)$,
we obtain a contradiction to the choice of $g$.

Next, we assume that the edge $v_rw_y$ belongs to $G$ but the edge $v_sw_x$ does not.
Since $a_r<a_s$, there is a vertex $w_z$ that is a neighbor of $v_s$ but not $v_r$.
Let $G'$ arise from $G$ by removing the edges $v_rw_x$ and $v_sw_z$,
and adding the edges $v_rw_z$ and $v_sw_x$,
and let $M'=(M\setminus \{ v_rw_x,v_sw_y\})\cup \{ v_sw_x,v_rw_y\}$.
By construction, $X$ is a minimum vertex cover of $G'$,
and $M'$ is a maximum matching in $G'$.
Since 
$s_1((G,M',X))=s_2((G,M',X))=s_3((G,M',X))=s_4((G,M',X))=0$, 
$s_5((G',M',X))<s_5(g)$, and
$s_6((G',M',X))=s_6(g)$,
we obtain a contradiction to the choice of $g$.

Finally, we assume that the edge $v_sw_x$ belongs to $G$ but the edge $v_rw_y$ does not.
Since $b_x<b_y$, there is a vertex $v_t$ that is a neighbor of $w_y$ but not $w_x$.
Let $G'$ arise from $G$ by removing the edges $v_rw_x$ and $v_tw_y$,
and adding the edges $v_rw_y$ and $v_tw_x$,
and let $M'=(M\setminus \{ v_rw_x,v_sw_y\})\cup \{ v_sw_x,v_rw_y\}$.
By construction, $X$ is a minimum vertex cover of $G'$,
and $M'$ is a maximum matching in $G'$.
Since 
$s_1((G,M',X))=s_2((G,M',X))=s_3((G,M',X))=s_4((G,M',X))=0$, 
$s_5((G',M',X))<s_5(g)$, and
$s_6((G',M',X))=s_6(g)$,
we obtain a contradiction to the choice of $g$.

Altogether, we obtain $s_5(g)=0$.

In all four cases, $s_6$ is not affected by the modifications.
By symmetry, this implies $s_6(g)=0$,
which completes the proof of the claim. 
\end{proof} 
As observed above, the three claims complete the proof.
\end{proof} 
Theorem \ref{theorem2} allows to reformulate the considered problems using network flows.

Therefore, 
let $a_1,\ldots,a_n,b_1,\ldots,b_m,\nu$, and $k$ be nonnegative integers with $k\leq \nu\leq \min\{ n,m\}$.
Let $(d_A,d_B)=((a_1,\ldots,a_n),(b_1,\ldots,b_m))$, 
and let $N(d_A,d_B,\nu,k)$
be the network $(D,c)$, where
\begin{itemize}
\item $D$ is a digraph with vertex set 
$$\{ s\}\cup \{ v_1,\ldots,v_n\}\cup \{ w_1,\ldots,w_m\}\cup \{ t\}$$
and arc set
\begin{eqnarray*}
&&\{ (s,v_i):i\in [n]\}\\
&\cup &\bigg\{ (v_i,w_j):(i,j)\in [n]\times [m]\setminus \Big\{ (i,j)\in [n]\times [m]:\big((i>k)\wedge (j>\nu-k)\big)\vee \big(i+j=\nu+1\big)\Big\}
\bigg\}\\
&\cup &\{ (w_j,t):j\in [m]\}
\end{eqnarray*}
and
\item $c:A(D)\to\mathbb{R}_{\geq 0}$ is a capacity function with
$$c((x,y))=
\begin{cases}
a_i-1 & \mbox{, if $x=s$, $y=v_i$, and $i\in [k]$,}\\
a_i & \mbox{, if $x=s$, $y=v_i$, and $i\in [n]\setminus [k]$,}\\
b_j-1 & \mbox{, if $x=w_j$, $y=t$, and $j\in [\nu-k]$,}\\
b_j & \mbox{, if $x=w_j$, $y=t$, and $j\in [m]\setminus [\nu-k]$, and}\\
1 & \mbox{, otherwise.}
\end{cases}
$$
\end{itemize}
Adding arcs $(v_i,w_j)$ to $D$ only if $i\leq k$ or $j\leq \nu-k$,
reflects that $\{ v_1,\ldots,v_k\}\cup \{ w_1,\ldots,w_{\nu-k}\}$
is a vertex cover in the graph $G$ from Theorem \ref{theorem2}.
Not adding the arcs $(v_i,w_j)$ with $i+j=\nu+1$, 
and reducing the capacities of the arcs $(s,v_i)$ and $(w_j,t)$ by $1$,
reflects that $\{ v_1w_{\nu},\ldots,v_{\nu}w_1\}$ 
is a matching in $G$.

\begin{theorem}\label{theorem3}
Let $a_1,\ldots,a_n,b_1,\ldots,b_m$, and $\nu$ be nonnegative integers with $\nu\leq \min\{ n,m\}$,
and let $(d_A,d_B)=((a_1,\ldots,a_n),(b_1,\ldots,b_m))$.

$(d_A,d_B)$ is a bipartite degree sequence 
that has a realization with matching number $\nu$
if and only if 
there is an integer $k$ with $0\leq k\leq \nu$ 
such that the network $N(d_A,d_B,\nu,k)$
has an $s$-$t$-flow of value $\sum\limits_{i=1}^na_i-\nu$.
\end{theorem}
\begin{proof}
If $(d_A,d_B)$ has a realization with matching number $\nu$,
then let $G$ and $k$ be as in Theorem \ref{theorem2}.
Let $N=(D,c)$ be the network $N(d_A,d_B,\nu,k)$.
Setting the flow values within $N$ to 
\begin{itemize}
\item $a_i-1$ on the arc $(s,v_i)$ for every $i$ in $[k]$,
\item $a_i$ on the arc $(s,v_i)$ for every $i$ in $[n]\setminus [k]$,
\item $b_j-1$ on the arc $(w_j,t)$ for every $j$ in $[\nu-k]$, 
\item $b_j$ on the arc $(w_j,t)$ for every $j$ in $[m]\setminus [\nu-k]$, and 
\item $1$ on every arc of $D$ that corresponds to an edge of $G$ 
that does not belong to the maximum matching $\{ v_iw_{\nu-i+1}:i\in [\nu]\}$ in $G$,
\end{itemize}
yields an $s$-$t$-flow in $N$ of value $\sum\limits_{i=1}^na_i-\nu$.

Conversely, let the integer $k$ with $0\leq k\leq \nu$ be such that the network $N=N(d_A,d_B,\nu,k)$
has an $s$-$t$-flow $f$ of value $\sum\limits_{i=1}^na_i-\nu$.
Note that the $s$-$t$-cut generated by $s$ has capacity $\sum\limits_{i=1}^na_i-\nu$,
which implies that $f$ is a maximum flow \cite{fofu}.
Since all capacities within $N$ are integral,
we may assume that $f$ has only integral values \cite{dafo}.
Let $G$ be the bipartite graph with partite sets $A=\{ v_1,\ldots,v_n\}$ and $B=\{ w_1,\ldots,w_m\}$
whose edge set consists of 
\begin{itemize}
\item the edges $v_iw_j$ for every arc $(v_i,w_j)$ of $D$ from $A$ to $B$ with $f((v_i,w_j))=1$, and
\item the edges in $\{ v_iw_{\nu-i+1}:i\in [\nu]\}$.
\end{itemize}
By construction, $G$ is a realization of $(d_A,d_B)$,
$M$ is a matching in $G$, and 
$X=\{ v_i:i\in [k]\}\cup \{ w_j:j\in [\nu-k]\}$ is a vertex cover in $G$.
Since $|M|=|X|$, the matching $M$ is a maximum matching in $G$,
which completes the proof.
\end{proof}
Our next goal is to reduce the existence of a flow as in Theorem \ref{theorem3}
to a cubic number of inequalities similarly as in the Gale-Ryser Theorem.
We use the Max-Flow-Min-Cut Theorem \cite{fofu},
and our approach is inspired by proofs of the Gale-Ryser Theorem 
using network flows.

Therefore, let $(d_A,d_B)$, $\nu$, $k$, and $N(d_A,d_B,\nu,k)$ be as above. Abbreviate $N(d_A,d_B,\nu,k)$ as $N$.
We consider an $s$-$t$-cut $C$ in $N$ generated by a set 
$\{ s\}\cup S_1\cup S_2\cup T_1\cup T_2$, where
\begin{eqnarray*}
S_1 &\subseteq & \{ v_1,\ldots,v_k\},\\
S_2 &\subseteq & \{ v_{k+1},\ldots,v_n\},\\
T_1 &\subseteq & \{ w_1,\ldots,w_{\nu-k}\},\mbox{ and}\\
T_2 &\subseteq & \{ w_{\nu-k+1},\ldots,w_m\}.
\end{eqnarray*}
In view of the structure of $D$, 
the capacity ${\rm cap}(C)$ 
is the sum of the capacities of 
\begin{itemize}
\item the arcs from $s$ to $\{ v_1,\ldots,v_n\}\setminus (S_1\cup S_2)$,
\item the arcs from $S_1\cup S_2$ to $\{ w_1,\ldots,w_n\}\setminus (T_1\cup T_2)$, and 
\item the arcs from $T_1\cup T_2$ to $t$.
\end{itemize}
Recall that every arc from $S_1\cup S_2$ to $\{ w_1,\ldots,w_n\}\setminus (T_1\cup T_2)$
has capacity $1$.
Note that there are 
no arcs in $D$ from $S_2$ to $\{ w_{\nu-k+1},\ldots,w_m\}\setminus T_2$, and that 
the arcs from $S_1$ to $\{ w_1,\ldots,w_{\nu-k}\}\setminus T_1$
contribute exactly $|S_1|\cdot |T_1|$ to ${\rm cap}(C)$,
that is, their contribution does not depend on the specific choice of $S_1$ and $T_1$ but only on the cardinalities of these sets.
This last observation implies that, if we fix the cardinalities of $S_1$ and $T_1$,
then minimizing the capacity ${\rm cap}(C)$ of the cut $C$ 
splits into the two completely independent tasks of
\begin{itemize}
\item minimizing the contribution to ${\rm cap}(C)$ of 
the arcs from $S_1$ to $\{ w_{\nu-k+1},\ldots,w_m\}\setminus T_2$, and
\item minimizing the contribution to ${\rm cap}(C)$ of 
the arcs from $S_2$ to $\{ w_1,\ldots,w_{\nu-k}\}\setminus T_1$.
\end{itemize}
We introduce some properties (1) to (6) that the cut $C$ may have,
and if $C$ has all these properties, then we call it {\it clean}.
\begin{itemize}
\item[(1)] If $v_i\in S_1$ and $v_j\in \{ v_1,\ldots,v_k\}\setminus S_1$, then $a_i\geq a_j$.
\item[(2)] If $w_i\in T_1$ and $w_j\in \{ w_1,\ldots,w_{\nu-k}\}\setminus T_1$, then $b_i\leq b_j$.
\item[(3)] If $v_i\in S_2$ and $v_j\in \{ v_{k+1},\ldots,v_n\}\setminus S_2$, then $a_i\geq a_j$.
\item[(4)] If $w_i\in T_2$ and $w_j\in \{ w_{\nu-k+1},\ldots,w_m\}\setminus T_2$, then $b_i\leq b_j$.
\end{itemize}
If $C$ has property (1), then there is some integer $a'$ such that $S_1$ contains
all $v_i$ in $\{ v_1,\ldots,v_k\}$ with $a_i>a'$,
no $v_i$ in $\{ v_1,\ldots,v_k\}$ with $a_i<a'$, and 
some $v_i$ in $\{ v_1,\ldots,v_k\}$ with $a_i=a'$.
Note that $a'$ is uniquely determined 
if there is some $v_i\in S_1$ and some $v_j\in \{ v_1,\ldots,v_k\}\setminus S_1$
with $a_i=a_j=a'$. This implies that the following property is well defined.
\begin{itemize}
\item[(5)] If $v_i\in S_1$, $v_j\in \{ v_1,\ldots,v_k\}\setminus S_1$, and $a_i=a_j=a'$, then $i>j$.
\end{itemize}
Similarly, 
if $C$ has property (3), then there is some integer $a''$ such that $S_2$ contains
all $v_i$ in $\{ v_{k+1},\ldots,v_n\}$ with $a_i>a''$,
no $v_i$ in $\{ v_{k+1},\ldots,v_n\}$ with $a_i<a''$, and 
some $v_i$ in $\{ v_{k+1},\ldots,v_n\}$ with $a_i=a''$.
Again, $a''$ is uniquely determined 
if there is some $v_i\in S_2$ and some $v_j\in \{ v_{k+1},\ldots,v_n\}\setminus S_2$
with $a_i=a_j=a''$.
\begin{itemize}
\item[(6)] If $v_i\in S_2$, $v_j\in \{ v_{k+1},\ldots,v_n\}\setminus S_2$, and $a_i=a_j=a''$, then $i>j$.
\end{itemize}

\begin{lemma}\label{lemma3}
Some $s$-$t$-cut in $N$ of minimum capacity is clean.
\end{lemma}
\begin{proof}
Let $C$ be a minimum $s$-$t$-cut in $N$ 
generated by the sets $S_1$, $S_2$, $T_1$, and $T_2$ as above.
For every $\ell\in [4]$, let $s_\ell(C)$ be the number of pairs $(v_i,v_j)$ violating property ($\ell$).
Furthermore, 
if $C$ has property (1), 
then let $s_5(C)$ be the number of pairs $(v_i,v_j)$ violating property (5),
and 
if $C$ has property (3), 
then let $s_6(C)$ be the number of pairs $(v_i,v_j)$ violating property (6).
If $C$ fails to have property (1) or (3),
then let $s_5(C)$ and $s_6(C)$ be $\infty$, respectively.
Let $s(C)=(s_1(C),s_2(C),s_3(C),s_4(C),s_5(C),s_6(C))$.

We assume that $C$ is chosen in such a way that $s(C)$ is lexicographically minimal.
In order to complete the proof, it suffices to show that $s(C)=(0,0,0,0,0,0)$.

\setcounter{claim}{0}

\begin{claim}\label{claim4}
$s_1(C)=s_2(C)=0$.
\end{claim}
\begin{proof}[Proof of Claim \ref{claim4}]
Suppose, for a contradiction, that $(v_i,v_j)$ violates property (1),
that is, $v_i\in S_1$, $v_j\in \{ v_1,\ldots,v_k\}\setminus S_1$, but $a_i<a_j$.
Since the number of outneighbors of $v_i$ and $v_j$ in 
$\{ w_1,\ldots,w_m\}\setminus (T_1\cup T_2)$ 
differs by at most one,
and the capacity of the arc $(s,v_i)$ is larger than the capacity of the arc $(s,v_j)$,
replacing $v_i$ within $S_1$ by $v_j$ leads to a cut $C'$ 
for which $s(C')$ is lexicographically smaller than $s(C)$,
which contradicts the choice of $C$.
This implies $s_1(C)=0$.
Note that $s_2(C')=s_2(C)$.
A completely symmetric argument implies $s_2(C)=0$,
which completes the proof of the claim.
\end{proof}

\begin{claim}\label{claim5}
$s_3(C)=s_4(C)=0$.
\end{claim}
\begin{proof}[Proof of Claim \ref{claim5}]
Suppose, for a contradiction, that $(v_i,v_j)$ violates property (3),
that is, $v_i\in S_2$, $v_j\in \{ v_{k+1},\ldots,v_n\}\setminus S_2$, but $a_i<a_j$.
By the ordering of $d_A$, we have $j<i$.
Similarly as above,
let $C'$ be the cut generated by replacing $v_i$ within $S_2$ by $v_j$.
We consider two cases.

First, we assume that $i\geq \nu+1$.
By construction, 
every outneighbor of $v_j$ in $\{ w_1,\ldots,w_m\}\setminus (T_1\cup T_2)$
is also an outneighbor of $v_i$,
and the capacity of the arc $(s,v_i)$ is at most the capacity of the arc $(s,v_j)$,
regardless of whether $j$ is at most $\nu$ or bigger.
We obtain the contradiction that either ${\rm cap}(C')<{\rm cap}(C)$
or ${\rm cap}(C')={\rm cap}(C)$ but $s(C')$ is lexicographically smaller than $s(C)$.

Next, we assume that $i\leq \nu$.
By construction, 
the number of outneighbors of $v_i$ in $\{ w_1,\ldots,w_m\}\setminus (T_1\cup T_2)$
and the number of outneighbors of $v_j$ in that set
differ by at most one,
and the capacity of the arc $(s,v_i)$ is strictly smaller than the capacity of the arc $(s,v_j)$.
We obtain the same contradiction as above,
which implies $s_3(C)=0$.

Note that $(s_1(C'),s_2(C'),s_4(C'))=(s_1(C),s_2(C),s_4(C))$.
A completely symmetric argument implies $s_4(C)=0$,
which completes the proof of the claim.
\end{proof}
At this point we have already established that $C$ satisfies properties (1) to (4).
Let $a'$ be as in the definition of property (5).

\begin{claim}\label{claim6}
$s_5(C)=0$.
\end{claim}
\begin{proof}[Proof of Claim \ref{claim6}]
Suppose, for a contradiction, that $(v_i,v_j)$ violates property (5),
that is, $v_i\in S_1$, $v_j\in \{ v_1,\ldots,v_k\}\setminus S_1$, 
$a_i=a_j=a'$, but $i<j$.
Note that the arcs $(s,v_i)$ and $(s,v_j)$ both have capacity $a'-1$.
We consider two cases.

First, we assume that $w_{\nu-i+1}\in T_2$ or $w_{\nu-j+1}\not\in T_2$.
In this case, 
the number of outneighbors of $v_j$ in $\{ w_1,\ldots,w_m\}\setminus (T_1\cup T_2)$
is at most 
the number of outneighbors of $v_i$ in that set.
If $C'$ is the cut generated by replacing $v_i$ within $S_1$ by $v_j$,
then $s_1(C')=s_2(C')=s_3(C')=s_4(C')=0$, 
${\rm cap}(C')\leq {\rm cap}(C)$, but $s_5(C')<s_5(C)$,
which is a contradiction.

Next, we assume that $w_{\nu-i+1}\not\in T_2$ and $w_{\nu-j+1}\in T_2$.
Since $\nu-i+1>\nu-j+1$, the ordering of $d_B$ implies $b_{\nu-i+1}\leq b_{\nu-j+1}$.
By property (4), we have $b_{\nu-i+1}\geq b_{\nu-j+1}$.
Altogether, we obtain $b_{\nu-i+1}=b_{\nu-j+1}$.
If $C''$ is the cut generated by replacing $v_i$ within $S_1$ by $v_j$,
and replacing $w_{\nu-j+1}$ within $T_2$ by $w_{\nu-i+1}$, 
then $s_1(C')=s_2(C')=s_3(C')=s_4(C')=0$, 
${\rm cap}(C')={\rm cap}(C)$, but $s_5(C')<s_5(C)$,
which is a contradiction, and completes the proof of the claim.
\end{proof}
Let $a''$ be as in the definition of property (6).

\begin{claim}\label{claim7}
$s_6(C)=0$.
\end{claim}
\begin{proof}[Proof of Claim \ref{claim7}]
Suppose, for a contradiction, that $(v_i,v_j)$ violates property (6),
that is, $v_i\in S_2$, $v_j\in \{ v_{k+1},\ldots,v_n\}\setminus S_2$, 
$a_i=a_j=a''$, but $i<j$.
Now, the arcs $(s,v_i)$ and $(s,v_j)$ both have a capacity in $\{ a''-1,a''\}$.
Let $C'$ be the cut generated by replacing $v_i$ within $S_2$ by $v_j$.
We consider four cases.

First, we assume that $i\geq \nu+1$,
which implies that the arcs $(s,v_i)$ and $(s,v_j)$ both have capacity $a''$.
Since $v_i$ and $v_j$ have the same outneighbors in 
$\{ w_1,\ldots,w_m\}\setminus (T_1\cup T_2)$,
we obtain
$s_1(C')=s_2(C')=s_3(C')=s_4(C')=s_5(C')=0$, 
${\rm cap}(C')={\rm cap}(C)$, but $s_6(C')<s_6(C)$,
which is a contradiction.

Next, we assume that $i\leq \nu$ and $j\geq \nu+1$,
which implies that the capacity of the arc $(s,v_i)$ is $a''-1$,
and the capacity of the arc $(s,v_j)$ is $a''$.
Since 
the number of outneighbors of $v_j$ in $\{ w_1,\ldots,w_m\}\setminus (T_1\cup T_2)$
is at most 
the number of outneighbors of $v_i$ in that set,
we obtain 
$s_1(C')=s_2(C')=s_3(C')=s_4(C')=s_5(C')=0$, 
${\rm cap}(C')\leq {\rm cap}(C)$, but $s_6(C')<s_6(C)$,
which is a contradiction.

Next, we assume that $j\leq \nu$ and that $w_{\nu-i+1}\in T_1$ or $w_{\nu-j+1}\not\in T_1$,
which implies that the arcs $(s,v_i)$ and $(s,v_j)$ both have capacity $a''-1$.
Again,
the number of outneighbors of $v_j$ in $\{ w_1,\ldots,w_m\}\setminus (T_1\cup T_2)$
is at most 
the number of outneighbors of $v_i$ in that set,
and we obtain the same contradiction as in the previous case.

Finally, we assume that $j\leq \nu$, $w_{\nu-i+1}\not\in T_1$, and $w_{\nu-j+1}\in T_1$.
Again, the arcs $(s,v_i)$ and $(s,v_j)$ both have capacity $a''-1$.
Since $\nu-i+1>\nu-j+1$, the ordering of $d_B$ implies $b_{\nu-i+1}\leq b_{\nu-j+1}$.
By property (2), we have $b_{\nu-i+1}\geq b_{\nu-j+1}$.
Altogether, we obtain $b_{\nu-i+1}=b_{\nu-j+1}$.
If $C''$ is the cut generated by replacing $v_i$ within $S_2$ by $v_j$,
and replacing $w_{\nu-j+1}$ within $T_1$ by $w_{\nu-i+1}$, 
then $s_1(C')=s_2(C')=s_3(C')=s_4(C')=s_5(C')=0$, 
${\rm cap}(C')={\rm cap}(C)$, but $s_6(C')<s_6(C)$,
which is a contradiction, and completes the proof of the claim.
\end{proof}
As observed above the four claims complete the proof.
\end{proof}
The following lemma already contains expressions 
similar to those in the Gale-Ryser Theorem.

\begin{lemma}\label{lemma4}
Let $S_1\subseteq \{ v_1,\ldots,v_k\}$ satisfy properties (1) and (5),
and 
let $S_2\subseteq \{ v_{k+1},\ldots,v_n\}$ satisfy properties (3) and (6).
The minimum capacity of an $s$-$t$-cut $C$ in $N$ generated by a set $X$
with $S_1=X\cap \{ v_1,\ldots,v_k\}$ and
$S_2=X\cap \{ v_{k+1},\ldots,v_n\}$ equals
\begin{eqnarray*}
\sum\limits_{i\in S_1\cup S_2}c((s,v_i))
+
\sum\limits_{j\in [m]}\min\Big\{c((w_j,t)),\big|N^-_D(w_j)\cap (S_1\cup S_2)\big|\Big\}.
\end{eqnarray*}
\end{lemma}
\begin{proof}
The term $\sum\limits_{i\in S_1\cup S_2}c((s,v_i))$ is the contribution to ${\rm cap}(C)$
of the arcs between $s$ and $\{ v_1,\ldots,v_k\}$.
Furthermore, if the vertex $w_j$ for some $j$ in $[m]$ belongs to $X$,
then its contribution to ${\rm cap}(C)$ is $c((w_j,t))$,
while, if $w_j$ does not belong to $X$,
then its contribution to ${\rm cap}(C)$ is $\big|N^-_D(w_j)\cap (S_1\cup S_2)\big|$.
Minimizing these independent contributions of the $w_j$ yields the stated expression.
\end{proof}
It is not difficult to make the expression 
$\big|N^-_D(w_j)\cap (S_1\cup S_2)\big|$
in the previous lemma slightly more explicit
exploiting the very regular structure of $D$.

\begin{theorem}\label{theorem4}
Let $a_1,\ldots,a_n,b_1,\ldots,b_m$, and $\nu$ be nonnegative integers 
with $\nu\leq \min\{ n,m\}$.

$((a_1,\ldots,a_n),(b_1,\ldots,b_m))$
is a bipartite degree sequence 
that has a realization with matching number $\nu$
if and only
there is some integer $k$ with $0\leq k\leq \nu$
for which the $(k+1)(n+1-k)$ inequalities of the form
\begin{eqnarray*}
\sum\limits_{i\in S_1\cup S_2}c((s,v_i))
+
\sum\limits_{j\in [m]}\min\Big\{c((w_j,t)),\big|N^-_D(w_j)\cap (S_1\cup S_2)\big|\Big\}
\geq \sum\limits_{i\in [n]}a_i-\nu,
\end{eqnarray*}
where 
\begin{itemize}
\item $D$ is the digraph of the network $N(d_A,d_B,\nu,k)$,
\item $S_1\subseteq \{ v_1,\ldots,v_k\}$ satisfies properties (1) and (5), and 
\item $S_2\subseteq \{ v_{k+1},\ldots,v_n\}$ satisfies properties (3) and (6),
\end{itemize}
are satisfied.
\end{theorem}
\begin{proof}
By Theorem \ref{theorem3} and by the Max-Flow-Min-Cut Theorem \cite{fofu}, 
$((a_1,\ldots,a_n),(b_1,\ldots,b_m))$ 
is a bipartite degree sequence 
that has a realization with matching number $\nu$
if and only if 
all $s$-$t$-cuts in the network $N(d_A,d_B,\nu,k)$ 
for some $k$ 
have capacity at least $\sum\limits_{i\in [n]}a_i-\nu$.
By Lemma \ref{lemma3}, there is a clean minimum $s$-$t$-cut $C$ in $N(d_A,d_B,\nu,k)$ 
generated by a set $X$.
Considering all $(k+1)(n+1-k)$ possible values for the cardinalities 
of the sets $X\cap \{ v_1,\ldots,v_k\}$ and $X\cap \{ v_{k+1},\ldots,v_n\}$,
using the fact that 
sets $S_1$ and $S_2$ as in the statement are uniquely determined by their cardinalities,
and using Lemma \ref{lemma4}
allows to generate the inequalities of the stated form
that are all satisfied if and only if the minimum capacity of an $s$-$t$-cut in $N(d_A,d_B,\nu,k)$
has at least the desired value.
\end{proof}
Note that the number of possible choices for the triple 
$(k,|S_1|,|S_2|)$ is 
$O(\nu^2 n)$, which is as most $O(n^3)$.

Our final result is that the set of realizable matching numbers forms an interval.

If $G$ is a bipartite graph with partite sets $A$ and $B$,
then $G'$ arises from $G$ by a {\it bipartite swap}
if there are vertices $v$ and $v'$ in $A$, 
and $w$ and $w'$ in $B$ 
such that $vw$ and $v'w'$ are edges of $G$ 
but $vw'$ and $v'w$ are not,
and $G'$ arises from $G$ by removing the edges $vw$ and $v'w'$, 
and adding the edges $vw'$ and $v'w$.
Clearly, $G$ and $G'$ are realizations of the same bipartite degree sequence.

\begin{theorem}\label{theorem4}
If $(d_A,d_B)$ is a bipartite degree sequence 
that has realizations with matching numbers $\nu_{\min}$ and $\nu_{\max}$, 
and $\nu$ is an integer with $\nu_{\min}\leq \nu\leq \nu_{\max}$,
then $(d_A,d_B)$ has a realization with matching number $\nu$.
\end{theorem}
\begin{proof}
Let $G_{\min}$ and $G_{\max}$ be realizations of $(d_A,d_B)$
with matching numbers $\nu_{\min}$ and $\nu_{\max}$, 
respectively.
It is a folklore fact that there is a sequence 
$G_0,\ldots,G_k$ of realizations of $(d_A,d_B)$
such that 
$G_0=G_{\min}$,
$G_k=G_{\max}$, and
$G_i$ arises from $G_{i-1}$ by a bipartite swap for every $i$ in $[k]$.
In fact, it follows from proofs of a bipartite version of the Havel-Hakimi Theorem \cite{ha,ha2}
using bipartite swaps that $G_{\min}$ and $G_{\max}$ can both be transformed 
to the same realization of $(d_A,d_B)$ using bipartite swaps,
and, hence, they can be transformed into each other.)

Now, let $i\in [k]$, and let $G_i$ arise from $G_{i-1}$
by removing the edges $vw$ and $v'w'$, 
and adding the edges $vw'$ and $v'w$.
Let $M$ be a matching in $G_{i-1}$.
If $vw,v'w'\not\in M$, then let $M'=M$,
if $vw\in M$ and $v'w'\not\in M$, then let $M'=M\setminus \{ vw\}$,
if $vw\not\in M$ and $v'w'\in M$, then let $M'=M\setminus \{ v'w'\}$, and
if $vw,v'w'\in M$, then let $M'=(M\setminus \{vw,v'w'\})\cup \{vw',v'w\}$.
By construction, $M'$ is a matching in $G'$, 
which implies $\nu(G_i)\geq \nu(G_{i-1})+1$.
By symmetry, we obtain $\nu(G_{i-1})\geq \nu(G_i)+1$,
that is, the matching numbers of consecutive graphs in the sequence $G_0,\ldots,G_k$
differ by at most one, which implies the existence of the desired realization.
\end{proof}
Theorem \ref{theorem4} implies that the set of all possible matching
numbers of realizations of a bipartite degree sequence $(d_A,d_B)$,
where $d_A$ has $n$ elements and $d_B$ has $m$ elements, 
can be determined in $O(n^4 m)$ time,
because $O(n^4)$ inequalities have to be checked,
each of which can be checked in $O(m)$ time. 
It seems an interesting problem to find a faster algorithm for this task.
For a degree sequence $d$,
one can study $\nu^{\cal B}_{\max}(d)$ and $\nu^{\cal B}_{\min}(d)$,
where ${\cal B}$ is the class of all bipartite graphs.
Note that the complexity of deciding the existence of a bipartite realization 
of a given degree sequence is unknown.


\end{document}